\documentclass[sn-mathphys]{sn-jnl}

\theoremstyle{thmstyleone}%
\newtheorem{theorem}{Theorem}
\newtheorem{proposition}[theorem]{Proposition}%
\newtheorem{lemma}{Lemma}

\theoremstyle{thmstyletwo}%
\newtheorem{example}{Example}%
\newtheorem{remark}{Remark}%

\theoremstyle{thmstylethree}%

\begin{document}

\title[The exact projective penalty method]{The Exact Projective Penalty Method for Constrained Optimization}


\author*[1,2]{\fnm{Vladimir} \sur{Norkin}}\email{vladimir.norkin@gmail.com}

\affil*[1]{\orgdiv{V.M.Glushkov Institute of Cybernetics}, \orgname{National Academy of Sciences of Ukraine}, \orgaddress{\street{Glushkov avenue 40}, \city{Kyiv}, \postcode{03178}, \country{Ukraine}}}

\affil[2]{\orgdiv{Faculty of Applied Mathematics}, \orgname{National Technical University ''Igor Sikorsky Kyiv Polytechnic Institute''}, \orgaddress{\street{Polytehnichna 14}, \city{Kyiv}, \postcode{03056}, \country{Ukraine}}}


\abstract{A new exact projective penalty method is proposed for the equivalent reduction of constrained optimization problems to nonsmooth unconstrained ones. In the method, the original objective function is extended to infeasible points by summing its value at the projection of an infeasible point on the feasible set with the distance to the projection. Beside Euclidean projections, also a pointed projection in the direction of some fixed internal feasible point can be used. The equivalence means that local and global minimums of the problems coincide. Nonconvex sets with multivalued Euclidean projections are admitted, and the objective function may be lower semicontinuous. The particular case of convex problems is included. The obtained unconstrained or box constrained problem is solved by a version of the branch and bound method combined with local optimization. In principle, any local optimizer can be used within the branch and bound scheme but in numerical experiments sequential quadratic programming method was successfully used. So the proposed exact penalty method does not assume the existence of the objective function outside the allowable area and does not require the selection of the penalty coefficient.      
}




\maketitle

\section{Introduction}\label{sec1}

The classical approach to the exact reduction of a constrained optimization problem to an unconstrained one consists in adding to the objective function some nonsmooth penalty term for the violation of constraints. The problem in this method consists in selecting the correct penalty scale. 

Convex nonsmooth exact penalty functions were first introduced in \cite{Eremin_1966, Eremin_1967, Zangwill_1967}, and have been studied, for example, in \cite{Bertsekas_1975, Boukari_Fiacco_1995, Burke_1991, Demyanov_2005, Di_Pillo_1994, Evtushenko_Zhadan_1990, Han_Mangasarian_1979}, and many others works. Recent advances in the exact penalty function method, applications, and references can be found in   \cite{Dolgopolik_Fominyh_2019, Dolgopolik_2020, Dolgopolik_2021, Laptin_Bardadym_2019, Zaslavski_2010, Zhou_Lange_2015}. In this approach, the problem lies in the correct choice of the penalty parameter $M$. Selection of  $M$  too large makes the unconstrained problem ill-conditioned. A different approach to transformation of constrained problem into unconstrained ones by exact discontinuous penalty functions was considered in \cite{Batukhtin_1993, Knopov_Norkin_2022}. 

In the present paper, we propose a new exact projective penalty method of equivalent reduction of constrained optimization problems to unconstrained ones. The equivalence means that local and global minimums of the problems and the corresponding objective function values at the minimums coincide. In the proposed method, the original objective function is extended to infeasible points by summing its value at the projection of an infeasible point on the feasible set with the distance to the projection. Nonconvex feasible sets with multivalued projections are admitted, and the objective function may be lower semicontinuous. The special case of minimization of a Lipschitz function on a convex constraint set is included.
In the latter case, beside Euclidean projection we propose to use
a centralized projection in the direction of some fixed feasible point.  
So the method does not assume the existence of the objective function outside the allowable area and does not require the selection of the penalty coefficient. The proposed exact projective penalty function, in particular, extends a constrained  Lipschitz function defined on a convex set to the infeasible points with preserving Lipschitz property. In this regard, the proposed penalty function is related to the classical Kirgszbraun problem \cite{Kirszbraun_1934} on the possibility to extend a constrained Lipschitz function (mapping) to the whole space with preserving the Lipschitz property and Lipschitz constant. We show that the Kirgszbraun’s extended function preserves only global minima of the original constrained minimization problem.

The exact projective penalty function method was conceptually introduced and tested in \cite{Norkin_2020} (and later studied in \cite{Galvan_2021, Norkin_2022}) and was motivated by the application of the smoothing method to constrained global optimization to avoid irregularities at bounders of the feasible set. Here we validate it for general convex and nonconvex constrained optimization problems. 

Compared to \cite{Galvan_2021, Norkin_2020}, where constraints are assumed convex, the contribution of the present paper includes:

Extension and validation of the exact projective  penalty method to a general nonconvex case, i.e. to problems with a nonconvex constraint set and a lower semicontinuous objective function;

Validation of the variant of the method with non Euclidean projection operation, which may be computationally much cheaper than the Euclidean one and may be applied effectively to some nonconvex problems;

Application of the method to constrained stochastic programming problems;

Illustration of the method with non-Euclidean projection on a number of 
small non-linearly constrained problems and on large linearly constrained ones;

Establishing a relation of the proposed exact projective penalty function with Kirszbraun’s theorem \cite{Kirszbraun_1934, McShane_1934} on the possibility to extend a constrained Lipschitz function to the whole space with preserving Lipschitz property.

The paper proceeds as follows.
In Section \ref{general} we validate a general exact projective penalty method. Section \ref{convex}
considers a convex case. In Section \ref{non-Euclid} we study a variant of the method with non-Euclidean projection.
In section \ref{stoch} we apply the method to stochastic programming problems.
Section \ref{Kirszbraun} establishes a relation of the method to the Kirszbraun problem. 
Section \ref{numerical} presents algorithmic implementation and results of numerical experiments.
And finally, Section \ref{conclusions} is devoted to discussion of results and conclusions.

\section{The exact projective penalty function method}
\label{general}
Let it be necessary to solve a constrained optimization problem:
\begin{equation}
\label{primary_problem_1}
f(x)\to \underset{x\in C\subseteq {{\mathbb{R}}^{n}}}{\mathop{\min }},
\end{equation}
where $f(x)$ is a lower semicontinuous (lsc) function  defined on a closed set $C\subseteq {{\mathbb{R}}^{n}}$; ${{\mathbb{R}}^{n}}$ is $n$-dimensional Euclidian space with norm $\left\| \cdot  \right\|$; for $x,y\in {{\mathbb{R}}^{n}}$ define $d(x,y)=\left\| x-y \right\|$ and the distance ${{d}_{C}}(x)$ from $x$ to $C$ as ${{d}_{C}}(x)={{\min }_{y\in C}}d(x,y)$. For example, the set $C$ may be given by some other lower semicontinuous function $g(x)$, $C=\left\{ x\in {{\mathbb{R}}^{n}}:\,\,g(x)\le 0 \right\}$.

There are several ways to reduce constrained problem (\ref{primary_problem_1}) to an equivalent nonsmooth unconstrained one.  

For example, if 
$C=\left\{ x\in\mathbb{R}^n \mid g_j(x)\le 0,\,j=1,...,J\,;\,\,h_k(x)=0,\,k=1,...,K \right\}$, 
then in the exact penalty function method the Lipschitz function $f(x)$ is replaced by 
$$F(x):=f\left( x \right)+M\left( \sum\nolimits_{j}{\max \left\{ 0,{{g}_{j}}(x) \right\}}+
\sum\nolimits_{k}\mid h_k (x) \mid \right)$$  
or by 
$$F(x):=f\left( x \right)+M{{\inf }_{y\in C}}\left\| y-x \right\|.$$
with a sufficiently large penalty parameter $M$ and then one considers the problem of unconditional optimization of $F(x)$ (see, e.g. \cite[Proposition 2.4.3]{Clarke_1990}, \cite[Theorem 18.2]{MGN_1987}, \cite[Proposition 9.68]{Rockafellar_Wets_1998}). Note that here it is assumed that functions $f,{{g}_{j}},{{h}_{k}}$ are defined over the whole space ${{\mathbb{R}}^{n}}$.

Remark that there may be exact discontinuous penalty functions \cite{Batukhtin_1993, Knopov_Norkin_2022}.

Let $f:X\to {{\mathbb{R}}^{1}}$ be a lower semicontinuous function on a convex closed set $X\subseteq\mathbb{R}^n$;
$$D=\left\{ x\in X\mid{{g}_{j}}(x)\le 0,\,j=1,...,J\,;\,\,{{h}_{k}}(x)=0,\,k=1,...,K \right\},$$
where ${{g}_{j}}:X\to {{\mathbb{R}}^{1}}$ are lower semicontinuous, ${{h}_{k}}:X\to {{\mathbb{R}}^{1}}$ are continuous on $X$; point ${{x}_{0}}\in X\cap D$; constant $A>f({{x}_{0}})$. Then the global optimization problem  $f(x)\to {{\min }_{x\in X\cap D}}$ is equivalent to the problem of global minimization of the discontinuous penalty function
\[
F(x)=\left\{ \begin{matrix}
   f(x), & x\in X\cap D,  \\
   A+M\left( \sum\nolimits_{j=1}^{J}{\max \left\{ 0,{{g}_{j}}(x) \right\}}+
	\sum\nolimits_{k=1}^{K}{\mid {{h}_{k}}(x) \mid} \right), & x\notin X\cap D;  \\
\end{matrix} \right.
\]    
on the set $X$, $F(x)\to {{\min }_{x\in X}}$, where $M\ge 0$.  

Obviously, the global minima of $F(x)$ on $X$ lie in the region $X\cap D$, where $F(x)=f(x)$, so they coincide with the global minima of $f(x)$ on $X\cap D$.

\bigskip
Now consider the following unconstrained minimization problem:
\begin{equation}
\label{unconstrained_problem_2}
F(x):={{\min }_{y\in \pi_C (x)}}f(y)+M\,{{d}_{C}}(x)\to {{\min }_{x\in {{\text{R}}^{n}}}},	
\;\;\;M>0,
\end{equation}
where $C\subset\mathbb{R}^n$ is a closed set; $f(y)$ is lsc function defined on $C$; 
$\pi_C(x)=\mbox{argmin}_{y\in C}d(x,y)$ is a (possibly multi-valued) projection of point $x$ on $C$; $d_C(x)=\inf_{y\in C}d(x,y)$; $d(x,y)=\|x-y\|$. 

The following statement describes properties of the distance function $d_C(x)$ and projections 
$\pi_C(x)$.

\begin{proposition}
\label{Prop_RW_1}
\cite[Example 1.20]{Rockafellar_Wets_1998}. (Properties of distance functions and projections). For any nonempty, closed set $C\subset {{\mathbb{R}}^{n}}$, the distance ${{d}_{C}}(x)$ of a point $x$ from $C$ depends continuously on $x$, while the projection ${{\pi }_{C}}(x)$, consisting of the points of $C$ nearest to $x$ is nonempty and compact. Whenever ${{y}^{k}}\in {{\pi }_{C}}({{x}^{k}})$ and ${{x}^{k}}\to x$, the sequence $\left\{ {{y}^{k}} \right\}$ is bounded and all its cluster points lie in ${{\pi }_{C}}(x)$, i.e., the mapping $x\to {{\pi }_{C}}(x)$ is compact valued and upper semicontinuous.
\end{proposition}

Due to Proposition \ref{Prop_RW_1}  problem (\ref{unconstrained_problem_2}) is well defined,
namely, $\pi_C(x)\subset C$ is a nonempty compact set, so there exists ${{y}_{x}}\in {{\pi }_{C}}(x)$ such that ${{\min }_{y\in \pi_C (x)}}f(y)=f({{y}_{x}})$. Remark that function $\varphi (x)={{\min }_{y\in \pi_C (x)}}f(y)$ is lower semicontinuous \cite[Proposition 21]{Aubin_Ekeland_1984}, and function ${{d}_{C}}(x)$ is continuous, so function $F(x)$ is lower semicontinuous.

In what follows we use the following observation.
\begin{lemma}
\label{Geometric_lemma_1}
(A geometric lemma, further properties of projections on nonconvex sets in $\mathbb{R}^{n}$). Let $C\subset {{\mathbb{R}}^{n}}$ be closed set and ${{y}_{x}}$ be a projection on $C$ of point $x\notin C$. Then the point ${{y}_{x}}$ is the unique common projection on $C$  of all points ${{x}_{\lambda }}=(1-\lambda )x+\lambda {{y}_{x}}$, $\lambda \in \left( 0,1 \right]$.
\end{lemma}
\begin{proof} 
For $\lambda \in \left( 0,1 \right]$ and ${{x}_{\lambda }}=(1-\lambda )x+\lambda {{y}_{x}}$ find  ${{y}_{\lambda }}\in {{\pi }_{C}}({{x}_{\lambda }})$. Then by properties of projections, $d({{x}_{\lambda }},{{y}_{\lambda }})\le d({{x}_{\lambda }},{{y}_{x}})$ and $d(x,{{y}_{\lambda }})\ge d(x,{{y}_{x}})$.  Denote $\varphi \in [0,\pi ]$ the angle between vectors $\left( {{y}_{x}}-{{x}_{\lambda }} \right)$ and  $\left( {{y}_{\lambda }}-{{x}_{\lambda }} \right)$. By the cosine theorem, 
\[\begin{array}{lcl}
  {{d}^{2}}(x,{{y}_{\lambda }})&=&{{d}^{2}}(x,{{x}_{\lambda }})+{{d}^{2}}({{x}_{\lambda }},{{y}_{\lambda }})-2d(x,{{x}_{\lambda }})d({{x}_{\lambda }},{{y}_{\lambda }})\cos (\pi -\varphi ) \\ 
 &=&{{d}^{2}}(x,{{x}_{\lambda }})+{{d}^{2}}({{x}_{\lambda }},{{y}_{\lambda }})+2d(x,{{x}_{\lambda }})d({{x}_{\lambda }},{{y}_{\lambda }})\cos (\varphi ). 
\end{array}\]
Then the inequality $d(x,{{y}_{\lambda }})\ge d(x,{{y}_{x}})$ jointly with the relation
\[\begin{array}{lcl}
   {{d}^{2}}(x,{{y}_{\lambda }})&=&{{d}^{2}}(x,{{x}_{\lambda }})+{{d}^{2}}({{x}_{\lambda }},{{y}_{\lambda }})+2d(x,{{x}_{\lambda }})d({{x}_{\lambda }},{{y}_{\lambda }})\cos \varphi \\ 
 &\le& {{\left( d(x,{{x}_{\lambda }})+d({{x}_{\lambda }},{{y}_{x}}) \right)}^{2}}={{d}^{2}}(x,{{y}_{x}}), 
\end{array}\]
gives $d(x,{{y}_{\lambda }})=d(x,{{y}_{x}})$. Suppose, ${{y}_{\lambda }}\ne {{y}_{x}}$. Then, for $\varphi \in \left\{ 0,\pi  \right\}$, ${{y}_{\lambda }}$ appears on the line going through $\left\{ x,{{y}_{x}} \right\}$, in these cases ${{y}_{\lambda }}$ is closer to $x$ than ${{y}_{x}}$, a contradiction.   If $\varphi \in \left( 0,\pi  \right)$, then $\mid \cos \varphi  \mid<1$  and we also come to a contradiction,
\[\begin{array}{lcl}
   {{d}^{2}}(x,{{y}_{\lambda }})&=&{{d}^{2}}(x,{{x}_{\lambda }})+{{d}^{2}}({{x}_{\lambda }},{{y}_{\lambda }})+2d(x,{{x}_{\lambda }})d({{x}_{\lambda }},{{y}_{\lambda }})\cos \varphi < \\ 
 &<&{{\left( d(x,{{x}_{\lambda }})+d({{x}_{\lambda }},{{y}_{x}}) \right)}^{2}}={{d}^{2}}(x,{{y}_{x}}),
\end{array}\]
that proves the lemma.
$\Box$
\end{proof}

The lemma exploits the property that a small sphere included in a larger one can touch the latter only at a single point. In case of noncovex set $C$, there can be several projections ${{y}_{x}}$ of  point $x\notin C$ on $C$. For a convex set $C$, the statement of the Lemma 1 is available in \cite{Galvan_2021}.

The following theorem establishes condition of equivalence of problems (\ref{primary_problem_1}) and (\ref{unconstrained_problem_2}).

\begin{theorem}
\label{Theorem_1}
(A general projective penalty function method). Let function $f$ be lower semicontinuous on a non-empty closed set $C$. Any $M>0$ is admitted. Then problems (\ref{primary_problem_1}) and (\ref{unconstrained_problem_2}) are equivalent, i.e., each local (global) minimum of one problem is a local (global) minimum of the other, and the optimal values of the problems in the corresponding minima coincide.
\end{theorem}
\begin{proof}
 Let ${{x}^{*}}\in C$ be a global minimum of (\ref{primary_problem_1}). Take an arbitrary $x\in {{\mathbb{R}}^{n}}$ and find ${{y}_{x}}\in \pi_C (x)$ such that $f({{y}_{x}})={{\min }_{y\in \pi_C (x)}}f(y)$. Then $F(x)\ge f({{y}_{x}})\ge f({{x}^{*}})=F({{x}^{*}})$,  thus ${{x}^{*}}$ is a global minimum of (\ref{unconstrained_problem_2}) with the same minimal value $f({{x}^{*}})$.

Let ${{x}^{**}}$ be a global minimum of (\ref{unconstrained_problem_2}). First let us show that ${{x}^{**}}\in C$, suppose the opposite, ${{x}^{**}}\notin C$. By Proposition \ref{Prop_RW_1}, there exists a compact projection set $\pi_C ({{x}^{**}})\subseteq C$ and ${{y}^{**}}\in \arg {{\min }_{y\in \pi_C ({{x}^{**}})}}f(y)$. Consider points ${{x}_{\lambda }}=(1-\lambda ){{x}^{**}}+\lambda {{y}^{**}}$, $\lambda \in [0,1]$. By Lemma \ref{Geometric_lemma_1}, in the Eucleadian space ${{\pi }_{C}}({{x}_{\lambda }})={{y}^{**}}$, i.e. projections of points ${{x}_{\lambda }}$, $\lambda \in (0,1]$,  coincide with ${{y}^{**}}$, the projection of ${{x}^{**}}$ on $C$.  Remark that  $d_C (x_\lambda )=d\left( {{x}_{\lambda }},{{y}^{**}} \right)=
(1-\lambda)d(x^{**},y^{**})$, $\lambda \in [0,1]$, and 
\[\begin{array}{lcl}
F({{x}_{\lambda }})&=&{{\min }_{y\in \pi_C ({{x}_{\lambda}})}}f(y)+M\,{{d}_{C}}({{x}_{\lambda }})=f({{y}^{**}})+M\,d({{x}_{\lambda }},{{y}^{**}})\\
&=&f({{y}^{**}})+M\,(1-\lambda)d({{x}^{**}},{{y}^{**}})=F(x^{**})-\lambda d({{x}^{**}},{{y}^{**}}).
\end{array}\]
From here it follows ${{x}^{**}}\in C$, otherwise $F({{x}_{\lambda }})<F({{x}^{**}})$, $d({{x}_{\lambda }},{{x}^{**}})=\lambda d({{y}^{**}},{{x}^{**}})$  for any $\lambda \in (0,1]$, a contradiction. But for $x\in C$ it holds ${{d}_{C}}(x)=0$ and $f(x)=F(x)\ge F({{x}^{**}})=f({{x}^{**}})$, hence ${{x}^{**}}$ is a global minimum of (\ref{primary_problem_1}).

	Let ${{x}^{*}}$ be a local minimum of (\ref{primary_problem_1}). Then there exist a neighborhood $V({{x}^{*}})$ of ${{x}^{*}}$ such that $f(x)\ge f({{x}^{*}})$  for all $x\in V({{x}^{*}})\cap C$. Let us show that ${{x}^{*}}$ is a local minimum of $F(x)$.  Since $\pi_C ({{x}^{*}})={{x}^{*}}$ and $\pi_C (\cdot )$ is upper semicontinuous, then for $V({{x}^{*}})$ there is an open vicinity $v({{x}^{*}})$ of ${{x}^{*}}$ such that $\pi_C (x)\subseteq V({{x}^{*}})$ for all $x\in v({{x}^{*}})$. Consider $x\in v({{x}^{*}})$ and find ${{y}_{x}}\in \pi_C (x)$ such that ${{d}_{C}}(x)=d(x,{{y}_{x}})$ and $f({{y}_{x}})=\underset{y\in \pi_C (x)}{\mathop{\inf }}\,f(y)$. Then for $x\in v({{x}^{*}})$ it holds
\[\begin{array}{lcl}
F(x)&=&\underset{y\in \pi_C (x)}{\mathop{\inf }}\,f(y)+M\,{{d}_{C}}(x)=f({{y}_{x}})+M\,d(x,{{y}_{x}})\\
&\ge& f({{y}_{x}})\ge f({{x}^{*}})=F(x^*).
\end{array}\]

If ${{x}^{**}}$ is a local minimum of (\ref{unconstrained_problem_2}), then, as was proven before, it is impossible that ${{x}^{**}}$ does not belong to $C$, i.e. ${{x}^{**}}\in C$. But since on $C$ functions $F(x)$ and $f(x)$ coincide, then ${{x}^{**}}$ is a local minimum of (\ref{primary_problem_1}).
\end{proof}

\begin{remark}
\label{Remark_2} 	Formulation (\ref{unconstrained_problem_2}) assumes finding all projection points ${{\pi }_{C}}(x)$ in a nonconvex constraint set $C$. Generally, this is an impractical task. If the non-convex feasible set $C={{C}_{1}}\cup ...\cup {{C}_{m}}$ is the union of a finite number of convex sets ${{C}_{1}},...,{{C}_{m}}$, then, of course, the original problem (\ref{primary_problem_1}) then can be split into $m$ problems of  form (\ref{primary_problem_1}) with convex constraint sets ${{C}_{i}}$, $i=1,...,m$. However, it also makes sense to solve single problem (\ref{unconstrained_problem_2}), since one can find all projections ${{\pi }_{{{C}_{1}}}}(x),...,{{\pi }_{{{C}_{m}}}}(x)$ and select among them ${{\pi }_{C}}(x)$, the closest to $x$. This can be done in parallel and not all projections may need to be found (exactly). The penalty function $F(x)=\underset{y\in {{\pi }_{C}}(x)}{\mathop{\inf }}\,f(y)+M\,{{d}_{C}}(x)$ then may be discontinuous and thus has to be solved by the appropriate method, e.g., by smoothing method \cite{ENW_1995, Knopov_Norkin_2022}. The other option is to find all projections ${{\pi }_{C}}(x)$ by the branch and bound method.
\end{remark}

\begin{remark}
\label{Remark_3} Theorem \ref{Theorem_1} implies also that both problems (\ref{primary_problem_1}) and (\ref{unconstrained_problem_2}) either have local (global) minima or do not have them.

\label{Remark_4} The projective penalty function method is extendable to those metric spaces where statements of Proposition \ref{Prop_RW_1} and Lemma \ref{Geometric_lemma_1} hold.
\end{remark}

\section{The case of convex constraint set \cite{Galvan_2021, Norkin_2020}}
\label{convex}
Let us consider problem (\ref{primary_problem_1}) in the case of a convex constraint set $C$. Then projection $\pi_C(x)$  is single-valued and problem (\ref{unconstrained_problem_2}) becomes:
\begin{equation}
\label{unconstrained_problem_3}
F(x):=f(\pi_C (x))+M\cdot{{d}_{C}}(x)\to {{\min }_{x\in {{\text{R}}^{n}}}},	
\;\;\;M>0,
\end{equation} 

If $C$ is a convex closed set, then function ${{d}_{C}}(x)$ is continuous \cite[Example 1.20]{Rockafellar_Wets_1998}, and the mapping ${{\pi }_{C}}(x)$ is single valued and continuous on ${{\mathbb{R}}^{n}}$ \cite[Example 2.25]{Rockafellar_Wets_1998} (even non-stretching \cite[Corrolary 12.20]{Rockafellar_Wets_1998}). If function $f$ is continuous (lower semicontinuous) on a convex closed set $C$, then function $F(x)$  in (\ref{unconstrained_problem_3}) is continuous (lower semicontinuous) on ${{\mathbb{R}}^{n}}$. Thus problem (\ref{unconstrained_problem_3}) is well defined and  the statement of Theorem \ref{Theorem_1} certainly holds true.

	If  $f(x)$ is convex and defined on a convex set $C$, then $F(x)$ in (\ref{unconstrained_problem_3}) is not necessarily convex on ${{\mathbb{R}}^{n}}$. 
For example, if $f(x)=x$, $C=\left\{ x\in {{\text{R}}^{1}}:x\le 0 \right\}$, $M<1$, then $F(x)=\min \{x,M\,x\}$.	

However, in the convex case, although the exact penalty function   in (\ref{unconstrained_problem_3}) can be nonconvex, it has no additional (false) stationary points, which are not stationary for the primary problem (\ref{primary_problem_1}) \cite[Proposition 4]{Galvan_2021}.
	
\begin{lemma}
\label{Lemma_2a}
(Lipschitz property of the exact penalty function, \cite[Lemma 1, where $M=1$]{Galvan_2021}). If  function $f$ is Lipschitzian with constant $L$ on a convex closed set $C$, then the function $F(x)$ defined by equality (\ref{unconstrained_problem_3}) is also Lipschitzian with constant $(L+2M)$ on the whole space ${{\mathbb{R}}^{n}}$.
\end{lemma}	

In \cite[Proposition 4]{Galvan_2021}, it was also shown that in conditions of Lemma \ref{Lemma_2a} all stationary by Clarke \cite{Clarke_1990} points of the Lipschitz function $F(x)$ in (\ref{unconstrained_problem_3}), i.e. $x\in {{\mathbb{R}}^{n}}$ such that $0\in {{\partial }_{Clarke}}F(x)$, are stationary for the constrained problem (\ref{primary_problem_1}).	

Let $f:C\to {{\mathbb{R}}^{1}}$ and the convex set $C\subseteq {{\mathbb{R}}^{n}}$ in (\ref{primary_problem_1}) has a representation:
$$C=\left\{ x\in\mathbb{R}^n\mid{{g}_{j}}(x)\le 0,\,j=1,...,J\,;\,\,{{h}_{k}}(x)=0,\,k=1,...,K \right\},$$
where functions  ${{g}_{j}}$ are continuous and convex, and ${{h}_{k}}$ are linear. Denote ${{\pi }_{C}}(x)$ the projection of point $x$ on the set $C$. For a simple set $C$ given by linear constraints, the problem of searching projection ${{\pi }_{C}}(x)$ is either solved analytically or reduced to a quadratic programming problem. The properties of the exact penalty function (outside the feasibility set) may depend on the form of the penalty term. For example, the performance of optimization methods applied to problem (\ref{unconstrained_problem_2}) depends on the choice and control of parameter   \cite{Galvan_2021}.  So, we introduce one more penalty function
\begin{equation}
\label{eqn_5}
F(x):=f\left( {{\pi }_{C}}(x) \right)+M\left( \sum\nolimits_{j=1}^{J}{\max \left\{ 0,{{g}_{j}}(x) \right\}}+\sum\nolimits_{k=1}^{K}{\mid {{h}_{k}}(x) \mid} \right),\;\;\;	M>0, 
\end{equation}
and consider the problem of unconstrained optimization:
\begin{equation}
\label{eqn_6}
F(x)\to \underset{x\in {{\mathbb{R}}^{n}}}{\mathop{\min }}\,.	
\end{equation}
Note that in (\ref{eqn_5}) function $f$   may not be defined outside the feasible domain $C$.

\begin{theorem}
\label{Theorem_4}
(A projective penalty function for a convex constraint set given by equalities and inequalities). Let function $f$ be lower semicontinuous on a non-empty closed convex set $C$. Then problems (\ref{primary_problem_1}) and (\ref{eqn_5})-(\ref{eqn_6}) are equivalent, i.e., each local (global) minimum of one problem is a local (global) minimum of the another, and the optimal values of the problems in the corresponding minima coincide.
\end{theorem} 
\begin{proof} 
Let ${{x}^{*}}\in C$ be a local minimum of problem (\ref{primary_problem_1}), i.e., for some neighborhood ${{V}_{1}}({{x}^{*}})$  of the point ${{x}^{*}}$, it is also a global minimum point of $f(x)$ on the set ${{V}_{1}}({{x}^{*}})\cap C$. Obviously, for any $x\in {{V}_{1}}({{x}^{*}})$, due to the non-stretching property of the projection operator ${{\pi }_{C}}(\cdot )$ onto a convex set \cite[Corrolary 12.20]{Rockafellar_Wets_1998}, it is satisfied ${{\pi }_{C}}(x)\in {{V}_{1}}({{x}^{*}})$ and, thus,
\[\begin{array}{lcl}
F(x)&=&f\left( {{\pi }_{C}}(x) \right)+M\left( \sum\nolimits_{j=1}^{J}{\max \left\{ 0,{{g}_{j}}(x) \right\}}+\sum\nolimits_{k=1}^{K}{\mid {{h}_{k}}(x) \mid} \right)\\
&\ge& f\left( {{\pi }_{C}}(x) \right)\ge f\left( {{x}^{*}} \right)=F({{x}^{*}}),
\end{array}\]
i.e., ${{x}^{*}}$  is a local minimum of the function $F$. 

	Let ${{x}^{**}}$ be a local minimum point of function $F$, i.e., for some neighborhood $V({{x}^{**}})\subset {{\mathbb{R}}^{n}}$  point ${{x}^{**}}$ is the global minimum of the function $F$ on the set $V({{x}^{**}})$. Let's show that ${{x}^{**}}\in C$. Assume the contrary, ${{x}^{**}}\notin C$, then
\[\begin{array}{lcl}	
F({{x}^{**}})&=&f\left( {{\pi }_{C}}({{x}^{**}}) \right)+M\left( \sum\nolimits_{j=1}^{J}{\max \left\{ 0,{{g}_{j}}({{x}^{**}}) \right\}}+\sum\nolimits_{k=1}^{K}{\mid {{h}_{k}}({{x}^{**}}) \mid} \right)\\
&>&f\left( {{\pi }_{C}}({{x}^{**}}) \right).
\end{array}\]
Denote, ${{x}_{\lambda }}=(1-\lambda ){{x}^{**}}+\lambda {{\pi }_{C}}({{x}^{**}})$. By Lemma \ref{Geometric_lemma_1}, it holds ${{\pi }_{C}}({{x}_{\lambda }})={{x}^{**}}$.  Let us consider a convex function 
\[\begin{array}{lcl}
   \Phi (\lambda )&=&F\left( {{x}_{\lambda }} \right)=f\left( {{\pi }_{C}}({{x}_{\lambda }}) \right)+M\left( \sum\nolimits_{j=1}^{J}{\max \left\{ 0,{{g}_{j}}({{x}_{\lambda }}) \right\}}+
	\sum\nolimits_{k=1}^{K}{\mid {{h}_{k}}({{x}_{\lambda }}) \mid} \right) \\ 
 &=&f\left( {{\pi }_{C}}({{x}^{**}}) \right)+M\left( \sum\nolimits_{j=1}^{J}{\max \left\{ 0,{{g}_{j}}({{x}_{\lambda }}) \right\}}+\sum\nolimits_{k=1}^{K}{\mid {{h}_{k}}({{x}_{\lambda }}) \mid} \right),\,\,\,\,\lambda \in [0,1]. 
\end{array}\]
Obviously, 
\[\begin{array}{lcl}
   \Phi (\lambda )&=&F\left( {{x}_{\lambda }} \right)\le (1-\lambda )F({{x}^{**}})+\lambda F({{\pi }_{C}}({{x}^{**}}))  \\ 
 &=&F({{x}^{**}})-\lambda \left( F({{x}^{**}})-F({{\pi }_{C}}({{x}^{**}}) \right) \\ 
 &=&F({{x}^{**}})-\lambda \left( F({{x}^{**}})-f({{\pi }_{C}}({{x}^{**}}) \right)<F({{x}^{**}}),\,\,\,\,\,\,\,\,\,\,\lambda \in \left( 0,1 \right]. 
\end{array}\]
For all sufficiently small $\lambda $, we have ${{x}_{\lambda }}\in V({{x}^{**}})$ and $F({{x}_{\lambda }})<F({{x}^{**}})$, i.e., we obtain a contradiction that ${{x}^{**}}$ is not a local minimum of the function $F$. In this way,  ${{x}^{**}}\in C$. For all $x\in V({{x}^{**}})\cap C$, it holds $f(x)=f({{\pi }_{C}}(x))=F(x)\ge F({{x}^{**}})=f({{x}^{**}})$, i.e., the point ${{x}^{**}}$ is also a local minimum point for $f$ on $C$ and $F({{x}^{**}})=f({{x}^{**}})$. The proof of the coincidence of global minima is carried out in a similar way. 
\end{proof}

\section{The variant of the method with non-Euclidean projection}
\label{non-Euclid}
The projection operation   may be computationally costly since it assumes solution of a quadratic optimization problem under a convex constraint. This can be a hard problem if constraints are nonlinear. So consider a variant of the method, which instead of projecting finds a root of a one-dimensional nonlinear equation. 

If some admissible point ${{x}_{0}}\in C$  is known, the exact penalty function can be constructed as follows. Let $x\notin C$ and $y(x)$ be the nearest to $x$ point of the set $C$ lying on the segment connecting ${{x}_{0}}$ and $x$. Let us define the mapping
\begin{equation}
\label{non-euclid_proj}
{{p}_{C}}(x)=\left\{ \begin{matrix}
   x, & x\in C,  \\
   y(x), & x\notin C,  \\
\end{matrix} \right. 
\end{equation}
and the penalty functions ${{r}_{C}}(x)=\left\| x-{{p }_{C}}(x) \right\|$ and 
$$F(x):=f({{p}_{C}}(x))+M{{r}_{C}}(x).$$ 

If $x_0\in\mathbb{R}^n$ and $g({{x}_{0}})<0$, then finding the projection ${{p}_{C}}(x)$ is reduced to finding the largest root of the one dimensional equation $\varphi (\lambda ):=g\left( {{x}_{0}}+\lambda (x-{{x}_{0}}) \right)=0$ on the interval $\lambda \in [0,1]$, which can be done efficiently by a one dimensional search.

The next two lemmas give sufficient conditions for mapping $p_C(x)$ to be 
continuous.
\begin{lemma}
\label{contin_centr_proj}
Let the set $C$ is bounded and convex, $x_0$ is an internal point in $C$. 
Then the projection mapping $p_C(x)$ (\ref{non-euclid_proj}) is continuous.
\end{lemma}
\begin{proof}
Let $y(x)$ be defined by as in (\ref{non-euclid_proj}). 
Function $l(x)=\|y(x)-x_0\|$ is concave near any $x\ne x_0$, 
so $l(x)$ is continuous near $x$. 
Hence, mapping $p(x)=l(x)(x-x_0)/\|x-x_0\|$ 
is continuous at $x\ne x_0$ and mapping $p_C(x)$ is continuous everywhere. 
\end{proof}
\begin{lemma}
\label{Lemma_4} (Continuity of the non-Euclidean projection on a set given by a monotone function). Let $C=\left\{ x\in \mathbb{R}_{+}^{n}:\,\,\,g(x)\le 0 \right\}$, where $g(x)$ is a strictly monotone continuous  function on  $\mathbb{R}_{+}^{n}$, i.e., if ${{x}_{1}}>{{y}_{1}},...,{{x}_{n}}>{{y}_{n}}$, then $g(x)>g(y)$. Assume that $g(0)<0$ and ${{\sup }_{\lambda \ge 0}}g(\lambda x)>0$ for any $x\in \mathbb{R}_{+}^{n}$, $x\ne 0$. Denote for each $x\in \mathbb{R}_{+}^{n}$, $x\ne 0$, the root ${{\lambda }_{x}}$ of the equation  $g(\lambda x)=0$. Then the projection mapping ${{p}_{C}}(x)=\min\{1,\lambda(x)\}x$ for $x\ne 0$ and
$p_C(0)=0$ is continuous on $\mathbb{R}_{+}^{n}$.
\end{lemma}
\begin{proof}
Let us consider a parametric mapping $y(\lambda,x)=\lambda(x)$, $\lambda\ge 0$,
function $\phi(\lambda,x)=g\left(y(\lambda,x) \right)$ and equation $\phi(\lambda,x)=0$. 
For each fixed $x^\prime\ne 0$ this equation has a unique solution 
$\lambda^\prime=\lambda(x^\prime)>0$, $\phi(\lambda(x^\prime),x^\prime)=0$. 
Locally, around point $(\lambda^\prime,x^\prime)\in\mathbb{R}^{n+1}_+$ function $\phi(\lambda,x)$ is
monotonic in the first variable under fixed the second one. By the implicit function
theorem \cite{Jittorntrum_1978, Kumagai_1980} the implicit function $\lambda(x)$ 
is continuous around 
$x^\prime$, i.e., it is continuous everywhere in $\mathbb{R}^n_+$ except $x=0$.
Then 
$$
p_C(x)=\left\{
\begin{array}{ll}
\min\{1,\lambda(x)\}x,&x\ne 0,\\
0,& x=0,
\end{array}
\right.
$$ 
is continuous in $\mathbb{R}^n_+$. 
\end{proof}

Consider the unconstrained optimization problem: 
\begin{equation}
\label{Unconstrained_probl_3}
F(x):=f({{p }_{C}}(x))+M\,{{r}_{C}}(x)\to {{\min }_{x\in 
{{\mathbb{R}}^{n}}}},	\;\;\;	M>0.
\end{equation}

The next theorem validates reduction of constrained problem (\ref{primary_problem_1}) to unconstrained one
(\ref{Unconstrained_probl_3}).
\begin{theorem}
\label{non-euclid_pen}
(An exact penalty function with a non-Euclidian projection). Let $f$ be some function on a non-empty closed set $C$. Then problems (\ref{primary_problem_1}) and (\ref{Unconstrained_probl_3}) are globally equivalent, i.e., if there is a global minimum of one problem, then  it is a global minimum of the other and objective function values at these minimums coincide. Moreover, any local minimum of problem (\ref{Unconstrained_probl_3}), if exists, is a local minimum of problem (\ref{primary_problem_1}). In the case when the constraint set $C$
satisfies conditions of Lemmas \ref{contin_centr_proj}, \ref{Lemma_4}, any local minimum of problem (\ref{primary_problem_1}), if exists, is a local minimum of problem (\ref{Unconstrained_probl_3}).
\end{theorem}
\begin{proof}
 If there is a global minimum ${x}^{*}\in C$ of problem (\ref{primary_problem_1}), then for any $x\in \mathbb{R}^{n}$ it holds
\begin{equation}
\label{eqn_4}
F(x)=f(p_{C}(x))+M\,{{r}_{C}}(x)=f\left( {{p }_{C}}(x) \right)\ge f({{x}^{*}})=F({{x}^{*}}).
\end{equation}

If there is a local (global) minimum ${{x}^{**}}$ of problem (\ref{Unconstrained_probl_3}), then for some neighborhood $V({{x}^{**}})\subset {{\mathbb{R}}^{n}}$ the point ${{x}^{**}}$ is the global minimum of the function $F$ on the set $V({{x}^{**}})$. Let us show that ${{x}^{**}}\in C$. Assume the contrary, ${{x}^{**}}\notin C$, then	
$$F({{x}^{**}})=f\left( {p}_{C}({x}^{**})\right)+M\left\| {x}^{**}-
{p}_{C}({x}^{**}) \right\|>f\left( {p}_{C}({x}^{**}) \right).$$
Denote ${x}_{\lambda}=(1-\lambda ){x}^{**}+\lambda {p}_{C}({x}^{**})$. Let us consider a function $\Phi (\lambda )=F\left( {{x}_{\lambda }} \right)$, $\lambda \in [0,1]$.  Obviously, 
\[\begin{array}{lcl}
   \Phi (\lambda )&=&F\left( {{x}_{\lambda }} \right)= 
	f(p_C(x_\lambda))+M\left\| x_\lambda-{p}_{{C}}(x_\lambda) \right\|\\
	&=&f(p_C(x^{**}))+M\left\| x_\lambda-{p }_{{C}}(x^{**}) \right\|\\
	&=&f({p }_{{C}}(x^{**}))+M\left\| (1-\lambda ){{x}^{**}}+\lambda {p}_{{C}}({{x}^{**}})-{p}_{{C}}(x^{**}) \right\|\\
	&=&f({p }_{{C}}(x^{**}))+M(1-\lambda )\left\| {{x}^{**}}-{p}_{{C}}(x^{**}) \right\|\\
	&=&F(x^{**})-\lambda M\left\| {{x}^{**}}-{p}_{{C}}(x^{**}) \right\|
	<F({{x}^{**}}),\,\,\,\,\,\,\,\,\,\,\lambda \in \left( 0,1 \right].  
\end{array}\]

For all sufficiently small $\lambda $, we have ${{x}_{\lambda }}\in V({{x}^{**}})$ and $F({{x}_{\lambda }})<F({{x}^{**}})$, i.e., we obtain a contradiction that ${{x}^{**}}$ is not a local minimum of the function $F$. In this way, ${{x}^{**}}\in C$. For all, $x\in V({{x}^{**}})\cap C$ it holds $f(x)=f({p}_{{C}}(x))=F(x)\ge F({{x}^{**}})=f({{x}^{**}})$, i.e., the point ${{x}^{**}}$ is also a local (global) minimum point for $f$ on $C$ and $F({{x}^{**}})=f({{x}^{**}})$.

If  ${{x}^{*}}$ is a local minimum of problem (\ref{primary_problem_1}), i.e., in some neighborhood $V({{x}^{*}})$ this point ${{x}^{*}}$ is a global minimum on the set $V({{x}^{*}})\cap C$. By assumption and due to Lemmas \ref{contin_centr_proj}, \ref{Lemma_4}, the mapping ${p }_{{C}}(x)$ is continuous. Therefore, there is a smaller neighborhood $W({{x}^{*}})\subseteq V({{x}^{*}})$ such that for any $x\in W({{x}^{*}})$ it holds ${{p}_{C}}(x)\in V({{x}^{*}})$. Hence, for any $x\in W({{x}^{*}})$,  inequality (\ref{eqn_4}) is true, which means that ${{x}^{*}}$ is a local minimum of problem (\ref{Unconstrained_probl_3}). The proof is complete.
\end{proof}

\begin{remark}
\label{Remark_6} In Theorem \ref{non-euclid_pen} we don't assume that the objective function
$f$ is lower semicontinuous, i.e., the theorem states that if one problems has a local (global) minimum, then the other one also has a local (global) minimum. If the objective function of problem (\ref{primary_problem_1}) is known to be lower semicontinuous, then under conditions of the theorem both problems
(\ref{primary_problem_1}) and (\ref{Unconstrained_probl_3})
have local (global) minimums. 
\end{remark}

\begin{example}
\label{Example_2} (Linear constraints with positive coefficients). Let 
$$C=\left\{ x\in \mathbb{R}_{+}^{n}:\,\,\sum\nolimits_{j=1}^{n}{{{a}_{ij}}{{x}_{j}}}\le {{b}_{i}},\,i=1,...,m \right\},$$ 
where all ${{a}_{ij}}\ge 0$, ${{b}_{i}}>0$, and ${{\max }_{1\le j\le n}}{{a}_{ij}}>0$. Then conditions of Lemma \ref{Lemma_4} are satisfied for $g(x)={{\max }_{1\le i\le m}}\sum\nolimits_{j=1}^{n}{\left( {{{a}_{ij}}}/{{{b}_{i}}}\; \right){{x}_{j}}}-1$.  Indeed, $g(0)=-1$,  ${{\lambda }_{\,x}}={{\left( \underset{1\le i\le m}{\mathop{\max }}\,\sum\nolimits_{i=1}^{n}{\left( {{{a}_{ij}}}/{{{b}_{i}}}\; \right){{x}_{j}}} \right)}^{-1}}$ is a continuous function at any $x\in \mathbb{R}_{+}^{n}$, $x\ne 0$. 
\end{example}

\begin{example}
\label{Example_3} (Nonconvex case). Let 
$$C=\left\{ x\in \mathbb{R}_{+}^{n}:\,\,\underset{1\le i\le m}{\mathop{\min }}\,\sum\nolimits_{j=1}^{n}{\left( {{{a}_{ij}}}/{{{b}_{i}}}\; \right){{x}_{j}}}\le 1 \right\},$$ 
where all ${{a}_{ij}}>0$, ${{b}_{i}}>0$. Then conditions of Lemma \ref{Lemma_4} are satisfied for $g(x)={{\min }_{1\le i\le m}}\sum\nolimits_{j=1}^{n}{\left( {{{a}_{ij}}}/{{{b}_{i}}}\; \right){{x}_{j}}}-1$.  Indeed, $g(0)=-1$,  ${{\lambda }_{\,x}}={{\left( \underset{1\le i\le m}{\mathop{\min }}\,\sum\nolimits_{i=1}^{n}{\left( {{{a}_{ij}}}/{{{b}_{i}}}\; \right){{x}_{j}}} \right)}^{-1}}$ is a continuous function at any $x\in \mathbb{R}_{+}^{n}$, $x\ne 0$. 
\end{example}

\begin{remark}
\label{Remark_8} In conditions of Lemmas \ref{contin_centr_proj}, \ref{Lemma_4} one can take the exact penalty function in the composite form 
\begin{equation}
\label{func_rem_8}
F(x)=f\left( {{p}_{C}}\left( {{\pi }_{\mathbb{R}_{+}^{n}}}(x) \right) \right)+{{M}}\left\| x-{{\pi }_{\mathbb{R}_{+}^{n}}}(x) \right\|+{{M}}\left\| x-{{p}_{C}}\left( {{\pi }_{\mathbb{R}_{+}^{n}}}(x) \right) \right\|,	{M}>0,
\end{equation}
where we first find the Euclidean projection ${{\pi }_{\mathbb{R}_{+}^{n}}}(x)$ of point $x$ on the set $\mathbb{R}_{+}^{n}$ and then find the non-Euclidean projection ${{p}_{C}}\left( {{\pi }_{\mathbb{R}_{+}^{n}}}(x) \right)$ of  ${{\pi }_{\mathbb{R}_{+}^{n}}}(x)$ on $C$. Both projection operators are continuous, so their composition is continuous too. Due to construction of the projection operators ${{\pi }_{\mathbb{R}_{+}^{n}}}$ and ${{p}_{C}}$,  there cannot exist stationary points outside the feasible set.  By Theorems \ref{Theorem_1}, 
\ref{non-euclid_pen}, function (\ref{func_rem_8}) is the exact penalty function for problem: $f(x)\to {{\min }_{\left\{ x\in \mathbb{R}_{+}^{n}:g(x)\le 0 \right\}}}$. 
\end{remark}

\begin{remark}
\label{Remark_9} In conditions of Lemma \ref{Lemma_4} the mapping ${{p}_{{{C}_{0}}}}(x)={{\lambda }_{x}}x$ projects point $x$ on the set ${{C}_{0}}=\left\{ x\in \mathbb{R}_{+}^{n}:\,g(x)=0 \right\}$; it is continuous for $x\ne 0$ and thus can be used for construction of the exact penalty function for equality-constrained problems, 
$$F(x)=f\left( {{p}_{{{C}_{0}}}}\left( {{\pi }_{\mathbb{R}_{+}^{n}}}(x) \right) \right)+{{M}_{1}}\left\| x-{{\pi }_{\mathbb{R}_{+}^{n}}}(x) \right\|+{{M}_{2}}\left\| x-{{p}_{{{C}_{0}}}}\left( {{\pi }_{\mathbb{R}_{+}^{n}}}(x) \right) \right\|,	{{M}_{1}},{{M}_{2}}>0$$.
\end{remark}

\begin{example}
\label{Block-wise}
(Block-wise optimization problem).
Consider the problem
\[
f(x)+\sum_{j=1}^{n}c_{j}(x)y_j\rightarrow \min_{x\in X,y\ge 0},
\]
\[
g_i(x)+\sum_{j=1}^{n}a_{ij}(x)y_j\le b_i,\;\;\;i=1,...,m.
\]
Suppose there is $x_0\in X$ such that $g_i(x_0)\le b_i$, $i=1,...,m,$ then it can be rewritten in the following form
\begin{equation}
\label{block_probl}
f(x)+h(x)\rightarrow\min_{x\in X:g(x)\le 0},
\end{equation}
where
\[
g(x)=\max_{1\le i\le m}(g_i(x)- b_i),
\]
\begin{equation}
\label{function_h}
h(x)=\min_{y\ge 0}\left\{\sum_{j=1}^{n}c_{j}(x)y_j
\mid\sum_{j=1}^{n}a_{ij}(x)y_j\le b_i-g_i(x), \;\;\;i=1,...,m. \right\}
\end{equation}

The constraint $g(x)\le 0$ cannot be removed from the problem (\ref{block_probl}) by just adding the penalty term proportional to $\max{\{0,g(x)\}}$ to the objective function 
$f(x)+h(x)$ since $h(x)$ does not exist under $g(x)>0$. 
However, this constraint can be removed by means of the projective penalty method.
If $g_i(x)$ are convex on the bounded convex set $X$ and $g(x_0)<0$ for some $x_0\in X$, 
then the composite projection operator $p_{\{x:g(x)\le 0\}}(\pi_X(x))=q(x)$ is continuous 
on $X$ and the original problem is equivalent to the following one:
\[
f(q(x))+h(q(x))+\|x-q(x)\|\rightarrow \min_{x\in X}.
\]

Remark that if functions $g_i(x)$ are linear on a box set $X$, then both projections
 $y=\pi_X(x)$ and $p_{\{x:g(x)\le 0\}}(y)$ can be found in a closed form. For example,
let $X=\mathbb{R}_+^K$, $g_i(x)=\sum_{j=1}^k \alpha_{ij}x_j$, and $\alpha_{ij}>0$, $b_i>0$
for all $i,j$, 
then we can take $x_0=0$ and $q(x)=\min{\{1,\lambda(x)\}}x$, where 
$$\lambda(x)=\min_{1\le i\le m}\left(b_i/\sum_{j=1}^k\alpha_{ij}x_j\right).$$
\end{example}	

\section{The exact projective penalty function for stochastic programming problems}
\label{stoch}
A typical stochastic programming problem has the following form:
\begin{equation}
\label{stoch_progr}
f(x)=\mathbb{E}\phi(x,\omega)\longrightarrow\min_{x\in C},
\end{equation}
where $C\subseteq \mathbb{R}^n$ is a closed convex set, $\omega$ is an elementary event of some probability space $\left(\Omega,\Sigma,P\right)$, $\mathbb{E}$ is the the sign of mathematical expectation over probability measure $P$. The exact projective penalty function
for problem (\ref{stoch_progr}) has the form:
\begin{equation}
\label{pen_stoch_progr}
F(x):=\mathbb{E}\phi(\pi_C(x),\omega)+M d_C(x)=\mathbb{E}\bar{\phi}(x,\omega),
\end{equation}  
where $\bar{\phi}(x,\omega)=\phi(\pi_C(x),\omega)+M d_C(x)$. If function 
$\phi(\cdot,\omega)$ is Lipschitz continuous in the first variable on the closed convex set $C$ with integrable Lipschits constant $l(\omega)$, then, by Lemma \ref{Lemma_2a}, function
$\bar{\phi}(\cdot,\omega)$ is Lipschits continuous with constant 
$\bar{l}(\omega)=l(\omega)+2M$ and the penalty function (\ref{pen_stoch_progr}) 
is also Lipschitz continuous with constant $L=\mathbb{E}l(\omega)+2M$.

\section{Relation of the exact penalty function to the classical Kirszbraun problem}
\label{Kirszbraun}

Let $C\subset {{\mathbb{R}}^{n}}$ and $f:C\to {{\mathbb{R}}}$ be a Lipschitz function with constant $L$, i.e.
\[
	\mid f({{x}_{1}})-f({{x}_{2}}) \mid\le L\cdot d({{x}_{1}},{{x}_{2}}) \;\;\;\;\; \forall {{x}_{1}},{{x}_{2}}\in C,	\;\;d({{x}_{1}},{{x}_{2}})=\left\| {{x}_{1}}-{{x}_{2}} \right\|.
	\]
The Kirszbraun’s theorem \cite{Kirszbraun_1934} states that function $f$ can be extended to the whole space ${{\mathbb{R}}^{n}}$ with preserving the Lipschitz constant. One explicit formula for this extension is given in the next theorem. Further results and context can be found in \cite{Azagra_2021}.
\begin{theorem}
\label{McShane}
\cite{McShane_1934}. Let $C\subset {{\mathbb{R}}^{n}}$ and $f:C\to {\mathbb{R}}$ be a Lipschitz function with constant $L$. Then function
	$$\Phi (x):={{\inf }_{y\in C}}\left( f(y)+M\cdot d(x,y) \right)$$ 	
is Lipscitzian with constant $M$  and for $M\ge L$ it holds $f(x)=\Phi (x)$ for all $x\in C$.
\end{theorem}
\begin{proof}
 Indeed,
	$$\mid \Phi ({{x}_{1}})-\Phi ({{x}_{2}}) \mid\le 
	M{{\sup }_{y\in C}}\mid d({{x}_{1}},y)-d({{x}_{2}},y) \mid\le M\left\| {{x}_{1}}-{{x}_{2}} \right\|.$$
For any $x\in C$, 
	$$\Phi (x)={{\inf }_{y\in C}}\left( f(y)+M\cdot d(x,y) \right)\le {{\inf }_{y\in C}}\left( f(x)+M\cdot d(x,y) \right)\le f(x).$$
For any $x\in C$ and $M\ge L$, 
\[
\begin{array}{lcl}
   \Phi (x)&=&{{\inf }_{y\in C}}\left( f(y)+M\cdot d(x,y) \right)\ge {{\inf }_{y\in C}}\left( f(x)-L\cdot d(x,y)+M\cdot d(x,y) \right) \\
	&=&f(x)+{{\inf }_{y\in C}}(M-L)d(x,y)=f(x).  
\end{array}
\]
\end{proof}

Let us remark other properties of $\Phi (x)$:

If $f(\cdot )$ is continuous on a compact set $C$, then $\Phi (x)$ is continuous on ${{\mathbb{R}}^{n}}$;

If $f(\cdot )$ is convex on a compact convex set $C$, then $\Phi (x)$ is convex on ${{\mathbb{R}}^{n}}$.

\begin{lemma}
\label{lemma_first}
If ${{x}^{*}}$ is a global minimize of an arbitrary function $f$ on $C$, then ${{x}^{*}}$ is a global minimizer of $\Phi $ on ${{\mathbb{R}}^{n}}$:
$$f({{x}^{*}})={{\inf }_{y\in C}}f(y) \mbox{  implies  } \Phi ({{x}^{*}})=f({{x}^{*}})={{\inf }_{y\in {{\mathbb{R}}^{n}}}}\Phi (y).$$
\end{lemma}
\begin{proof} 
Let ${{x}^{*}}$ be a global minimize of $f$ on $C$, i.e., $f(y)\ge f({{x}^{*}})$ for all $y\in C$. Then for all $y\in C$and $x\in {{\mathbb{R}}^{n}}$
$$f(y)+M\cdot d(x,y)\ge f({{x}^{*}})+M\cdot d(x,y).$$
From here for all  $x\in {{\mathbb{R}}^{n}}$, it holds
\[
\begin{array}{lcl}
   \Phi (x)&=&{{\inf }_{y\in C}}\left( f(y)+M\cdot d(x,y) \right)\ge {{\inf }_{y\in C}}\left( f({{x}^{*}})+M\cdot d(x,y) \right)  \\ 
 &=&f({{x}^{*}})+{{\inf }_{y\in C}}M\cdot d(x,y)\ge f({{x}^{*}}). 
\end{array}
\]
In particular, $\Phi ({{x}^{*}})\ge f({{x}^{*}})$. Besides,
	$$\Phi ({{x}^{*}})={{\inf }_{y\in C}}\left( f(y)+M\cdot d({{x}^{*}},y) \right)\le \left( f({{x}^{*}})+M\cdot d({{x}^{*}},{{x}^{*}}) \right)=f({{x}^{*}}).$$
Hence, $\Phi (x^{*})=f({x}^{*})=\inf_{y\in \mathbb{R}^{n}}\Phi (y)$. 
\end{proof}

\begin{lemma}
\label{lemma_second}
Suppose $C$ is closed. If ${x}'$ is a global minimizer of $\Phi (x)$ on ${{\mathbb{R}}^{n}}$, then ${x}'\in C$.
\end{lemma}
\begin{proof}
Suppose the opposite, ${x}'\notin C$. By closedness of $C$, there is $\varepsilon >0$ such that $d({x}',y)\ge \varepsilon >0$ for all $y\in C$. Let us take $\delta <M\cdot \varepsilon $ and find ${y}'\in C$ such that  
$$\Phi ({x}')={{\inf }_{y\in C}}\left( f(y)+M\cdot d({x}',y) \right)\ge f({y}')+M\cdot d({x}',{y}')-\delta.$$
From here,
\[
	\begin{array}{lcl}
   \Phi ({y}')&=&{{\inf }_{y\in C}}\left( f(y)+M\cdot d({y}',y) \right)\le f({y}')+M\cdot d({y}',{y}') \\ 
 &=&f({y}')\le \Phi ({x}')-M\cdot d({x}',{y}')+\delta \le \Phi ({x}')-M\cdot \varepsilon +\delta <\Phi ({x}'),  
\end{array} 
\]
a contradiction with global optimality of ${x}'$. 
\end{proof}

\begin{theorem}
\label{Kirsbraun_glob_opt}
Let $f$ be Lipschitzian on a closed set $C\subset {{\mathbb{R}}^{n}}$ with constant $L$, $M\ge L$. Then global minimums of $f(x)$ on $C$ coincide with global minimums of  the Kirszbraun function $\Phi (x)={{\inf }_{y\in C}}\left( f(y)+M\cdot d(x,y) \right)$ on ${{\mathbb{R}}^{n}}$. 
\end{theorem} 
	
	Proof. By Lemma \ref{lemma_first} each global minimum of  $f(x)$ on $C$ is a global minimizer for $\Phi (x)$ on ${{\mathbb{R}}^{n}}$. Conversely, let ${x}'$ be a global minimum of  $\Phi (x)$ on ${{\mathbb{R}}^{n}}$. By Lemma \ref{lemma_second}, ${x}'\in C$. By Theorem \ref{McShane}, $\Phi (x)=f(x)$ on $C$. So ${x}'$ is a global minimum of $f(x)$ on $C$. $\Box$
	
\begin{remark}
\label{Remark_10} The proposed exact projective penalty function (3) extends function $f(x)$ on $C\subset {{\mathbb{R}}^{n}}$  to $F(x)$ on $\mathbb{R}^n$ such that $f(x)=F(x)$ on $C$:
	$$F(x):=f({{\pi }_{C}}(x))+M\cdot d(x,{{\pi }_{C}}(x)),	\;\;M>0,$$
where  ${{\pi }_{C}}(x)=\arg {{\min }_{y\in C}}d(x,y)$. 

	Function $F(x)$ has the following properties:
	
	$F(x)=f({{\pi }_{C}}(x))+M\cdot d(x,{{\pi }_{C}}(x))\ge {{\inf }_{y\in C}}\left( f(y)+M\cdot d(x,y) \right)=\Phi (x)$  for $x\in {{\mathbb{R}}^{n}}$;
	
	$F(x)=f(x)$ for $x\in C$;
	
	If $f(x)$ is Lipschitzian with constant $L$ on a closed convex set $C$, then $F(x)$ is Lipscitzian on ${{\mathbb{R}}^{n}}$ with constant $\left( L+2M \right)$;
	
	Global and local minimums of $f(x)$ on $C$ are global and local minimums of $F(x)$ on ${{\mathbb{R}}^{n}}$ and conversely,  global and local minimums of $F(x)$ on ${{\mathbb{R}}^{n}}$ are global and local minimums of $f(x)$ on $C$.
\end{remark}

\begin{remark}
\label{Remark_11} In many cases it is much easier to calculate $F(x)$ than $\Phi (x)$ since calculation of $F(x)$ is reduced to finding the projection ${{\pi }_{C}}(x)$.

In the optimization context, both functions $F(x)$ and $\Phi (x)$can be used for replacement of the constrained optimization problem ${{\min }_{x\in C}}f(x)$ by the unconstrained ones,
${{\min }_{x\in {{\mathbb{R}}^{n}}}}F(x)$ and ${{\min }_{x\in {{\mathbb{R}}^{n}}}}\Phi (x)$. The first reformulation is advantageous, because it preserves local and global minimums of the original problem. The other reformulation preserves with guarantee only global minimums but for convex problems it also preserves convexity.
\end{remark}

\section{Numerical experiments}
\label{numerical}
In \cite{Norkin_2020}, the exact projective penalty method (\ref{unconstrained_problem_3}) was tested on
a number of test examples under box constraints, and in \cite{Galvan_2021} it was tested on small
optimization problems with linear constraints. Below we illustrate the performance of the method with non-Euclidean projection (\ref{non-euclid_proj}) on test problems with nonlinear and large linear  constraints. Since the resulting problem (\ref{Unconstrained_probl_3}) may be nonconvex, for its solution we use a version of the branch and bound (B\&B) method
combined  with sequential quadratic programming algorithm (under box constraints) implemented in the Matlab optimization toolbox. 
.

The structure of the {\it procedure} for solving problem (\ref{primary_problem_1}) is as follows.

{\it First}, we extract from $C$ (or add) a box constraint 
$$X=\{x=(x_1,...,x_n)^T\in \mathbb{R}^n: a_i\le x_i\le b_i, i=1,...,n \},$$ 
so the problem becomes
\begin{equation}
\label{primary_problem_2}
f(x)\to \underset{x\in C\cap X}{\mathop{\min }}.
\end{equation}

{\it Second}, we find a point $x^0\in X$ that lies in the interior of $C$. 
If $C=\{x\in\mathbb{R}^n:g(x)\le 0\}$, then for this we (approximately) solve the problem:
\begin{equation}
\label{feasibility_test}
g(x)\rightarrow \min_{x\in X}.
\end{equation}

{\it Third}, we form the exact penalty problem
\begin{equation}
\label{exact_penalty_problem_2}
F(x):=f(p_C(\pi_X(x)))+\|p_C(\pi_X(x))-\pi_X(x)\|+\|x-\pi_X(x)\|\rightarrow \min_{x\in X}.
\end{equation}

{\it Forth}, to solve problem (\ref{exact_penalty_problem_2}) (and (\ref{feasibility_test})), we apply the following branch and bound algorithm.

{\it The Branch \& Bound algorithm.}

{\it Initialization.} Set initial partition of ${\cal P}_0=\{X\}$, select a random starting point $\tilde{x}^0\in X$ and apply some local minimization algorithm $A$ to problem (\ref{exact_penalty_problem_2}).
As result we find a better point $\bar{x}^0\in X$ such that $F(\bar{x}^0)<F(\tilde{x}^0)$.
Set B\&B iteration number $k=0$. Set tollerances $\epsilon>0$ and $\delta>0$.

{\it B\&B iteration.} Suppose at some iteration $k$ we have partition ${\cal P}_k=\{X_i, i=1,...,N_k\}$ of the set  
$X=\cup_{i=1}^{N_k}X_i$ consisting of smaller boxes $X_i$. For each $X_i$ there is known
a feasible point $\bar{x}^i\in X_i$ and the value $F(\bar{x}^i)$, 
$V_k=\min_{1\le i\le N_k}F(\bar{x}^i)$. Set 
${\cal P}_{k+1}=\emptyset$.

For each such set $X_i\in {\cal P}_k$
choose  a random starting point $\tilde{x}^i$ and apply some local minimization algorithm $A$ to problem $\min_{x\in X_i}F(x)$ to find a better point $\bar{\bar{x}}^i\in X_i$,
$F(\bar{\bar{x}}^i)<F(\tilde{x}^i)$. 

If values $F(\bar{x}^i)$ and $F(\bar{\bar{x}}^i)$
are sufficiently different, say $\mid F(\bar{x}^i)-F(\bar{\bar{x}}^i)\mid\ge\epsilon$, or
points $\bar{\bar{x}}^i$ and $\bar{x}^i$ are sufficiently different, 
$\|\bar{\bar{x}}^i-\bar{x}^i\|\ge\delta$, we subdivide the box 
$X_i=X^\prime_i\cup X^{\prime\prime}_i$ into two subboxes
$X^\prime_i$ and $X^{\prime\prime}_i$ in such a way that 
$\bar{x}^i\in X^\prime_i$ and $\bar{\bar{x}}^i\in X^{\prime\prime}_i$. 
In this case partition ${\cal P}_{k+1}$ is updated by adding
successors $X^\prime_i$ and $X^{\prime\prime}_i$, i.e. 
${\cal P}_{k+1}:={\cal P}_{k+1}\cup X^\prime_i\cup X^{\prime\prime}_i$.

Otherwise, i.e. if values $F(\bar{x}^i)$ and $F(\bar{\bar{x}}^i)$ and points
$\bar{\bar{x}}^i$ and $\bar{x}^i$ are close, the set $X_i\ni \bar{x}^i$ goes unchanged 
to the updated partition ${\cal P}_{k+1}:={\cal P}_{k+1}\cup X_i$.

When all elements $X_i\in {\cal P}_k$ are checked, i.e. the new partition 
${\cal P}_{k+1}$ with elements $X_i, i=1,...,N_{k+1}$ and points $\bar{x}_i\in X_i$
has been constructed, we calculate the record achieved value 
$V_{k+1}=\min_{1\le i\le N_{k+1}} F(\bar{x}^i)$.

{\it Check for stop.} If the progress of the B\&B method becomes small, e.g. $V_k-V_{k+1}$ 
(or $V_{k-1}-V_{k+1}$, etc.) is sufficiently small,
then STOP, otherwise set $k:=k+1$ and continue branching and bounding iterations.  

\begin{remark}
Objective function values $F(\bar{x}^i)$ give upper bounds for the optimal values
$F^*_i=\min_{x\in X_i}F(x)$. If there are known a lower bounds $L_i\le F^*_i$, then the 
subsets $X_i\in{\cal P}_k$ such that $L_i\ge V_k$ can be safely ignored, i.e.
excluded from the current partion ${\cal P}_k$. 
Heuristically, if some set $X_i$ remains unchanged during several B\&B iterations,
it can be ignored in the future iterations.
\end{remark} 

\begin{remark}
The objective function in (\ref{exact_penalty_problem_2}) is nonsmooth, 
so for its minimization
it is advisable to apply methods of nondifferentiable optimization.
However for its local optimization, one may use within the B\&B framework any 
(even possibly not converging)
algorithms, which are able to improve the initial objective function value.
Indeed, experiments show that well developed smooth optimization algorithms 
like sequential quadratic programming ones
are also applicable for this purpose.
\end{remark} 
\begin{table}
\label{table1}
\caption{Testing the method (with non-Euclidean projection) on non-linear constrained problems from \cite{Hock_Schittkowski_1981}.
The columns of the table have the following meaning:
1) Test number from \cite{Hock_Schittkowski_1981};
2) Type of nonlinear constraints (inequalities or equalities);
3) Number of variables and nonlinear constraints;
4) Reference minimum value from \cite{Hock_Schittkowski_1981};
5) Achieved value by the exact penalty method;
6) Accuracy of the constraint satisfaction;
7) Number of penalty function calculation;
8) Number of constraint function calculation. 
}
\begin{tabular}{rrrrrrrr}
\hline
1	    &2	  &3	     &4	         &5	          &6	        &7	    &8	 \\
Test  &con	&var(con)&F\_ref	   &F\_min	    &G\_accur	  &F\_nmb	&G\_nmb \\
\hline
HS18	&ineq	&2(2)	   &5.0000e+0	 &5.0000e+0	  &3.7983e-8	&1180.2	&41.4 \\
HS20	&ineq	&2(3)	   &3.8199e+1  &3.8199e+1	  &2.0823e-6	&1015.4	&84.3 \\
HS34	&ineq	&3(2)	   &-8.3403e-1 &-8.1394e-1	&5.7009e-5	&60976	&614.3 \\
HS39	&eq		&4(2)    &-1.0000e+0 &-9.1689e-1	&2.4936e-8	&101450	&1392.2 \\
HS66	&ineq	&3(2)	   &5.1816e-1  &5.2572e-1	  &1.4229e-4	&74906	&541.5 \\
HS77	&eq		&5(2)	   &2.4151e-1  &3.8722e-1	  &1.2861e-7	&147830	&1902.7 \\
HS110	&box	&10(0)   &-4.5778e+1 &-4.5778e+1	&1.4901e-8	&2844.7	&44 \\
HS118	&ineq	&15(29)  &6.6482e+2  &6.6522e+2	  &5.7099e-4	&337930	&491170 \\
\hline
\end{tabular}
\end{table}
\noindent
\begin{table}
\label{table2}
\caption{Testing the method (with non-Euclidean projection) on randomly generated large linearly constrained problems, see Example \ref{Block-wise}.}
\begin{tabular}{rrrrrrrr}
\hline
1	    &2	      &3	     &4	             &5	          &6	        \\
Test  &Con.type &Var(con) &Func.accur.   &Constr.accur.&Func.calc.\\
\hline
1	    &lin-ineq	&10(20)	 &7.7766e-07	   &2.8790e-08	&1.2357e+03	\\
2   	&lin-ineq	&10(50)	 &5.6041e-07     &2.6211e-08	&9.3530e+02 \\
3   	&lin-ineq	&50(50)	 &2.9869e-06     &6.1793e-08	&6.3220e+03 \\
4   	&lin-ineq	&50(100) &1.3239e-06     &1.4488e-07	&6.1552e+03 \\
5   	&lin-ineq	&100(100)&8.1691e-06     &8.1228e-07  &1.4868e+04	\\
6   	&lin-ineq	&200(100)&3.0028e-06     &1.5685e-08	&2.3766e+04 \\
7   	&lin-ineq	&500(100)&5.1253e-06     &7.5100e-07	&1.2367e+05 \\
\hline
\end{tabular}
\end{table}
Table 7.1 presents some results concerning the performance of the exact projective penalty method on nonlinear constrained problems from 
\cite{Hock_Schittkowski_1981}. 
First an internal feasible point is found by solving (\ref{feasibility_test}), then it is used for calculation of non-Euclidean projections 
$p_C(\pi_X(x))$ in (\ref{exact_penalty_problem_2}). 
Equality constraints $g(x)=0$ are approximated by inequality ones $-\epsilon\le g(x)\le\epsilon$,
$\epsilon=10^{-10}$. For all tests the number of B\&B iterations is limited
to 10, i.e., the number of partitions is less or equal to 1024.
Each problems is solved ten times, the performance results are averaged.

Table 7.2 presents results of testing the method on randomly
generated problems as described in Example \ref{Block-wise}.
First, we randomly generate a linear programming problem with positive coefficients,
solve it, and obtain a reference optimal value.
Then we  subdivide variables into two groups. For fixed values of the first group variables, we solve
problem (\ref{function_h}) with respect the second group of variables.
As a result we obtain a nonsmooth convex optimization problem (\ref{block_probl}) with linear inequality constraints with respect to the first group variables. Column 3 of Table 7.2 indicates the number of variables and constraints
in the reduced problem. The latter problem is then solved by the exact projective 
penalty method with projections in the closed form as in Example
\ref{Block-wise}. The accuracy of the obtained solution and constraint 
satisfaction are indicated in columns 4, 5. The last column 6 shows
the number of the objective function calculations, i.e.
the number linear programming solver calls.
The results are averaged over ten runs for each test. 

We can make the following preliminary conclusions from the numerical experiments.

First, numerical experiments confirm theoretical results on 
the possibility of exact reduction of constrained optimization problems to 
unconstrained ones by the exact projective penalty method.

Second, the non-Euclidean projection as well as Euclidean one can be effectively 
used in the method.

Third, local optimization methods designed for solving smooth problems can be used
for solving nonsmooth problems in combination with the branch and bound method.

Forth, the exact projective penalty method can effectively solve
small nonlinear constrained problems and large linearly constrained ones.

Fifth, optimization problems with nonlinear equality constraint are hard
for the exact projective penalty method.

\section{Discussion and conclusions}
\label{conclusions}
In the paper, we equivalently reduce general constrained optimization problems to nonsmooth unconstrained ones without losing local and global minimums. In the proposed method, the original objective function is extended to infeasible area by summing its value at the projection of an infeasible point on the feasible set with the distance to the projection (or with other penalty terms). Nonconvex feasible sets with multivalued projections are admitted, and the objective function may be lower semicontinuous. So the method does not assume the existence of the objective function outside the allowable area and does not require the selection of the penalty coefficient. The special case of convex problems is included. However, the transformed problem may become nonconvex. So the preferable area of application of the method includes nonlinear problems with convex constraints, problems where the objective function is not defined outside the feasible set, and nonconvex global optimization problems with monotonic constraints. 

The discussed exact projective  penalty method (\ref{unconstrained_problem_3}) was first proposed conceptually in \cite{Norkin_2020} and later studied in \cite{Galvan_2021} for problems with convex constraints. In \cite{Norkin_2020, Norkin_2022}, a variant with non-Euclidean projection (\ref{non-euclid_proj}), (\ref{Unconstrained_probl_3}) was proposed too; also the method was numerically tested (in combination with the smoothing method \cite{Norkin_2020}) on a dozen of small (up to dimension 6) test functions from literature (under box constraints). The variant  (\ref{Unconstrained_probl_3}) with the non-Euclidean projection (\ref{non-euclid_proj}) is essential even in the convex case, since finding the Euclidean projection $\pi_C(\cdot)$  can be computationally costly relatively to the non-Euclidean one $p_C(\cdot)$. Besides, the non-Euclidean variant can be applied to some nonconvex problems with monotonic constraint functions as shown in Theorem \ref{non-euclid_pen}, Lemma \ref{Lemma_4}, Example \ref{Example_3}. Moreover, by Theorem \ref{non-euclid_pen}, to find some local minimums of a general nonconvex problem (\ref{primary_problem_1}) one may find some feasible point   and to solve nonconstrained problem (\ref{Unconstrained_probl_3}). 

In \cite{Galvan_2021},  method (\ref{unconstrained_problem_3})  (with the Euclidean projection on a convex constraint set) was theoretically validated and numerically tested on a large number of test problems of small dimensions. It was shown that Clarke's critical points of the transformed problem are critical for the original one. It was also shown that seemingly unnecessary penalty parameter $M$  in (\ref{unconstrained_problem_3}) may play an essential role in the speeding up the convergence and improving the accuracy of optimization procedures. 

Summarizing, in the present paper, we extended and validated the method for solving general constrained nonconvex optimization problems. 
Beside Euclidean projection operation, 
we included a non-Euclidean projection option to the construction of the method, 
and theoretically and numerically validated the modified method. The non-Euclidean projection
is reduced to finding a root of a one-dimensional nonlinear equation.
This suggestion is especially useful in case of large linear inequality constraints
where the projection can be found in a closed form.
Also, we showed applicability of the method for solving constrained stochastic optimization problems. We numerically demonstrated a possibility
to solve arising nonsmooth optimization problems by combination
of smooth optimization methods with the branch and bound technique. Further research may concern investigation of differential properties of projections ${{\pi }_{C}}(x)$, ${{p}_{C}}(x)$, ${{p}_{{{C}_{0}}}}(x)$,   and penalty functions in (\ref{unconstrained_problem_2}), (\ref{Unconstrained_probl_3}).
First step in this direction is made in Examples \ref{Example_2}, \ref{Example_3}
where projections are found in a closed analytical form. Numerical implementation of the projection exact penalty method for solving general nonconvex constrained global optimization problem based on the projection toward just feasible points (as in Theorem 4) will be the subject of further research.

\bigskip
{\bf Acknowledgment}. The work was supported by a stipend of Volkswagen Foundation (2022) and grant 2020.02/0121 of the National Research Foundation of Ukraine. The author thanks Professor Georg Pflug from Vienna University and Professor Alois Pichler from Chemnitz University of Technology for valuable discussions of the exact penalty function methods. In particular, Professor A.Pichler pointed out on the relation of exact penalty functions to the classical Kirszbraun problem. The author also thanks Vienna University and Chemnitz University of Technology for the hospitality during spring and summer 2022 and the possibility to complete this work.

\section*{Declarations}

The work was supported by stipend 9C090 of Volkswagen Foundation (2022) for fulfillment the project 
''Stochastic Global Optimization Methods for Solving Machine Learning Problems''
and by grant 2020.02/0121 of the National Research Foundation of Ukraine for performing the project 
''Analytical methods and machine learning in control theory and
decision-making in conditions of conflict and uncertainty.''

My manuscript has no associated data.




\end{document}